\documentclass[11pt, reqno]{amsart}
\usepackage{hyperref}
\usepackage{amsmath}
\usepackage{amscd}
\usepackage{amsthm}
\usepackage{amssymb}
\usepackage{latexsym}
\usepackage{eufrak}
\usepackage{euscript}
\usepackage{epsfig}
\usepackage{graphics}
\usepackage{array}
\usepackage{enumerate}
\usepackage{dsfont}
\usepackage{color}
\usepackage{wasysym}
\usepackage{graphicx}
\usepackage{amsfonts}
\usepackage{mathrsfs}
\usepackage{dsfont}
\usepackage{indentfirst}
\usepackage{esint}

\newtheorem{thm}{Theorem}[section]

\newtheorem{lemma}[thm]{Lemma}

\newtheorem*{claim}{Claim}

\theoremstyle{definition}
\newtheorem{definition}[thm]{Definition}
\newtheorem{example}[thm]{Example}
\newtheorem{remark}[thm]{Remark}

\newtheorem{conj}[thm]{Conjecture}

\newcommand{\Hmm}[1]{\leavevmode{\marginpar{\tiny%
$\hbox to 0mm{\hspace*{-0.5mm}$\leftarrow$\hss}%
\vcenter{\vrule depth 0.1mm height 0.1mm width \the\marginparwidth}%
\hbox to 0mm{\hss$\rightarrow$\hspace*{-0.5mm}}$\\\relax\raggedright
#1}}}

\begin{document}
\title[An upper bound for eigenvalue ratios]{An optimal dimension-free upper bound \\for eigenvalue ratios}
\author{Shiping Liu}
\address{Department of Mathematical Sciences, Durham University, DH1 3LE Durham, United Kingdom}
\email{shiping.liu@durham.ac.uk}
\begin{abstract}
On a closed weighted Riemannian manifold with nonnegative Bakry-\'{E}mery Ricci curvature, it is shown that the ratio of the $k$-th to the first eigenvalue of the weighted Laplacian is dominated by $641k^2$. While improving the previously known exponential upper bound, the order of $k$ here is optimal. This answers an open question of Funano. Our approach further proves affirmatively a conjecture of Funano and Shioya asserting a dimension free upper bound for eigenvalue ratios on a compact finite-dimensional Alexandrov space of nonnegative curvature.


\end{abstract}

\maketitle

\section{Introduction}
Let $(M,\mu)$ be a closed weighted Riemannian manifold, where $\mu$ is a Borel probability measure of the form $d\mu=e^{-V}d\text{vol}_M$, $V\in C^2(M)$ and $\text{vol}_M$ stands for the Riemannian volume measure of $M$. In this paper, we study the eigenvalues of the corresponding weighted Laplacian $\Delta_{\mu}$, which can be listed with multiplicity as below (see e.g. \cite{Setti98}),
\begin{equation*}
0=\lambda_0(M,\mu)<\lambda_1(M, \mu)\leq \lambda_2(M,\mu)\leq\cdots\leq\lambda_k(M, \mu)\leq\cdots\nearrow\infty.
\end{equation*}
We prove the following theorem.
\begin{thm}\label{thm:Main}
 For any closed weighted Riemannian manifold $(M,\mu)$ of nonnegative Bakry-\'{E}mery Ricci curvature and any natural number $k$, we have
\begin{equation}\label{eq:MainUpper}
\lambda_k(M,\mu)\leq\left(\frac{16e}{e-1}\right)^2k^2\lambda_1(M,\mu).
\end{equation}
\end{thm}

For $(M,\mu)$ as stated in the above theorem, it is known from the work of Cheng \cite{Cheng75} and Li-Yau \cite{LY80} (see also Setti \cite{Setti98}) on estimating eigenvalues by diameter that there exists a numeric constant $C$ such that
\begin{equation}\label{eq:dimensionUpper}
\lambda_k(M,\mu)\leq Cn^2k^2\lambda_1(M,\mu),
\end{equation}
where $n$ is the dimension of $M$. The first dimension-free estimate of these eigenvalue ratios was discovered by Funano and Shioya \cite{FS2013}
through investigating spectral characterizations of L\'{e}vy families. They proved that there exists a constant $C_k$ depending only on $k$
such that
\begin{equation}\label{eq:FunanoShioya}
\lambda_k(M,\mu)\leq C_k\lambda_1(M,\mu),
\end{equation}
and formulated the following conjecture
(see \cite[Conjecture 6.11]{FS2013} and also \cite[Conjecture 6.6]{Funano2013}).
\begin{conj}[Funano and Shioya]\label{conj:FS} For any natural number $k$, there exists a positive constant $C_k$ depending only on $k$ such that if $X$ is a compact finite-dimensional Alexandrov space of nonnegative curvature, then
\begin{equation}
\lambda_k(X)\leq C_k\lambda_1(X),
\end{equation}
where $\lambda_k(X)$ is the $k$-th non-zero eigenvalues of the corresponding Laplacian on $X$ \cite{KMS01}.
\end{conj}

Later, Funano \cite{Funano2013} proved a quantitative version of (\ref{eq:FunanoShioya}), showing that one can find a numeric constant $c>0$ such that $C_k$ in (\ref{eq:FunanoShioya}) can be taken as $$C_k=e^{ck}.$$ Funano further asked for the right order of $\lambda_k(M,\mu)/\lambda_1(M,\mu)$ in $k$, especially whether this ratio can be dominated by a polynomial function of $k$ (see \cite[Question 6.3]{Funano2013}).

Theorem \ref{thm:Main} answers Funano's question. In fact, the order of $k$ in (\ref{eq:MainUpper}) is optimal. This can be seen from the following examples.
\begin{example}\label{example}
 The eigenvalues of the circle $T_a^1$ of length $a$ are
\begin{equation*}
\left\{\frac{4\pi^2k^2}{a^2}: k=0,1,2,\ldots\right\},
\end{equation*}
where each non-zero eigenvalue has multiplicity $2$.
\end{example}

This examples shows the optimality of the order of $k$ in (\ref{eq:MainUpper}) for one dimensional manifolds.
\begin{example}\label{exampleThinTorus}
For $n\geq 2$, consider the $n$ dimensional "thin" torus $T_a^n:=\mathbb{R}^n/\Gamma_a^n$, where 
\begin{equation*}
\Gamma_a^n:=(a\mathbb{Z})^{n-1}\times \frac{1}{a^{n-1}}\mathbb{Z},
\end{equation*}
and $a\in (0,1)$. The dual lattice $(\Gamma_a^n)^*$ of $\Gamma_a^n$ is 
\begin{equation*}
(\Gamma_a^n)^*:=\{\gamma^*\in \mathbb{R}^n: \langle \gamma^*, \gamma\rangle\in \mathbb{Z}, \,\,\forall \gamma\in \Gamma\}=\left(\frac{1}{a}\mathbb{Z}\right)^{n-1}\times a^{n-1}\mathbb{Z}.
\end{equation*}
Then all the eigenvalues of $T_a^n$ are $\{4\pi^2|\gamma^*|^2, \,\,\forall\, \gamma^*\in (\Gamma_a^n)^*\}$  (see e.g. \cite[Section II.2]{Chavel}, \cite[Section 4.E.2]{GHLRieGeo}). Listing them monotonically with multiplicity, we have
\begin{align*}
\lambda_0(T_a^n)=0,\,\,\,& \lambda_{2m-1}(T_a^n)=\lambda_{2m}(T_a^n)=4\pi^2m^2a^{2(n-1)}, \,\,1\leq m\leq \left\lfloor\frac{1}{a^n}\right\rfloor, \\ &\lambda_{2\left\lfloor\frac{1}{a^n}\right\rfloor+1}(T_a^n)=\lambda_{2\left\lfloor\frac{1}{a^n}\right\rfloor+2}(T_a^n)=4\pi^2\frac{1}{a^2}.
\end{align*}
\end{example}
This example indicates that, for any natural number $n\geq 2$, any absolute constant $C>0$, and any number $\epsilon>0$, there exist an $n$ dimensional torus $T_a^n$ where $0<a<(((9C)^{1/\epsilon}+1)/2)^{-1/n}$, and a natural number $k=2\lfloor1/a^n\rfloor+1$, such that
\begin{equation}
\frac{\lambda_k(T_a^n)}{\lambda_1(T_a^n)}=\frac{1}{a^{2n}}\geq \frac{k^2}{9}>Ck^{2-\epsilon}.
\end{equation}
This shows the optimality of the order of $k$ in (\ref{eq:MainUpper}) for manifolds of arbitrary dimension.

It is worth to mention here Weyl's asymptotic formula (see e.g. \cite{Chavel}, \cite{Grigoryan06}). It reads
\begin{equation}
 \lambda_k(M,\mu) \sim c(n)\left(\frac{k}{\mu(M)}\right)^{\frac{2}{n}} \text{ as }k\rightarrow \infty,
\end{equation}
where $c(n)$ is a constant depending on dimension $n$.
Note further that while $Cn^2$ in (\ref{eq:dimensionUpper}) explodes as the dimension $n$ increases, the constant in our estimate (\ref{eq:MainUpper}) is fixed and smaller than $641$. We would also like to mention Funano and Shioya's observation that the nonnegativity of Ricci curvature is necessary (see \cite[Example 4.9]{FS2013}).

Moreover, our approach is extendable to a very general setting (see Section \ref{subsection:Exten}). In particular, we have the following result.
\begin{thm}\label{thm:AleNonnegative}
For any compact finite-dimensional Alexandrov space $X$ of nonnegative curvature and any natural number $k$, we have
\begin{equation}\label{eq:AlexUpper}
 \lambda_k(X)\leq \left(\frac{16e}{e-1}\right)^2k^2\lambda_1(X).
\end{equation}
\end{thm}
This verifies a strong version of the Conjecture \ref{conj:FS} of Funano and Shioya.
\subsection{An approach via Cheeger constant}

We prove Theorem \ref{thm:Main} by relating $\lambda_1(M,\mu)$ to $\lambda_k(M,\mu)$ via the Cheeger constant $h_1(M,\mu)$.

The first ingredient comes from a recent progress of Kwok et al. \cite{KLLGT2013} in theoretical computer science. In order to justify a well-known eigengap heuristic in the empirical performance of the spectral clustering algorithm, they proved a so-called \emph{improved Cheeger inequality} for finite graphs. They also indicated a corresponding inequality for closed Riemannian manifolds. But their constant depends on the dimension of the manifold. We prove the following dimension-free inequality for weighted Riemannian manifolds (see Section \ref{Section:proof}).

\begin{thm}\label{thm:ImproveCheeger}
On a closed weighted Riemannian manifold $(M,\mu)$, we have for any natural number $k$,
\begin{equation}\label{eq:ImprovedCheeger}
h_1(M,\mu)\leq 8\sqrt{2}k\frac{\lambda_1(M,\mu)}{\sqrt{\lambda_k(M,\mu)}}.
\end{equation}
\end{thm}
Note that when $k=1$, (\ref{eq:ImprovedCheeger}) reduces to Cheeger inequality \cite{Cheeger1970} with a larger constant. The proof of the improved Cheeger inequality for graphs depends on the fact that all eigenvalues of the normalized graph Laplacian are bounded from above by $2$. This is fortunately not needed in the continuous space case. Comparing with the graph case $h_1\leq 10\sqrt{2}(k+1)\lambda_1/\sqrt{\lambda_{k}}$ in \cite{KLLGT2013}, the constant $8\sqrt{2}$ in (\ref{eq:ImprovedCheeger}) is smaller. This is mainly due to a special argument valid in the continuous space case (see the Claim in the proof of Lemma \ref{lemma3.3}). We also have a higher order version of the inequality (\ref{eq:ImprovedCheeger}), see Theorem \ref{thm:ImprovedHigher}.


The second ingredient in the proof of Theorem \ref{thm:Main} is the following dimension-free Buser inequality \cite{Buser82} due to Ledoux \cite{Ledoux04}.

\begin{thm}\label{thm:LedouxBuser}
On a closed weighted Riemannian manifold $(M,\mu)$ of nonnegative Bakry-\'{E}mery Ricci curvature, we have
\begin{equation}\label{eq:LedouxBuser}
h_1(M,\mu)\geq \frac{e-1}{\sqrt{2}e}\sqrt{\lambda_1(M,\mu)}.
\end{equation}
\end{thm}
We comment that the constant here is larger than $1/6$ stated in \cite{Ledoux04}(see the formula (5.8) there).
The constant $(e-1)/\sqrt{2}e$ is obtained by restricting the proof of Lemma 5.1 in \cite{Ledoux04} to the particular case that Ricci curvature is nonnegative instead of bounded from below by a general number $-K$, $K\geq 0$.

Combining Theorems \ref{thm:ImproveCheeger} and \ref{thm:LedouxBuser} leads straightforwardly to a proof of Theorem \ref{thm:Main}.


The geometry and analysis of metric measure spaces, including discrete ones, have been extensively studied in recent years.
This topic benefited significantly from analogous ideas and results developed on Riemannian manifolds.
 Conversely, we see from the proof of Theorem \ref{thm:Main} that a deeper understanding of discrete spaces can also provide new insights and results on Riemannian manifolds. Moreover, there was recent progress in the spectral theory of Markov operators on general probability spaces based on results for discrete spaces: Miclo \cite{Miclo2013} confirmed a conjecture of Simon and H{\o}egh-Krohn based on the work of Lee, Oveis Gahran and Trevisan \cite{LGT2013} on finite graphs. This direction was further developed in \cite{Wang2014,Liu13}.

\subsection{Applications}
Theorem \ref{thm:Main} has important applications. We use it to improve the higher-order Buser-Ledoux inequality and the higher-order Gromov-Milman inequality established by Funano \cite{Funano2013} (see Theorems \ref{thm:higerBuserLedoux} and \ref{thm:higerGromovMilman} below). The latter inequality relates eigenvalues with the observable diameter introduced by Gromov \cite{Gromov99}. Actually, the higher-order Gromov-Milman inequality ((\ref{eq:GromovMilman}) below) obtained here can imply Cheng's dimension-dependent diameter estimates \cite{Cheng75} up to a universal constant.

Theorem \ref{thm:Main} also implies that the ratio of the $k$-th to the first isoperimetric constant (see Definition \ref{def:MultiwayCheeger}) can be bounded by a polynomial function of $k$, answering part of Question 6.3 in \cite{Funano2013}.


\section{Preliminaries}
Let $(M, \mu)$ be a closed weighted Riemannian manifold. The weighted Laplacian $\Delta_{\mu}$ is given by
\begin{equation*}
\Delta_{\mu}:=\Delta-\nabla V\cdot \nabla,
\end{equation*}
where $\Delta$ is the Laplace-Beltrami operator on $M$. For its basic spectral theory, we refer to Section 2 of \cite{Setti98}.
For any functions $f\in W^{1,2}(\mu)$, its Rayleigh quotient $\mathcal{R}(f)$ is defined as
\begin{equation*}
\mathcal{R}(f):=\frac{\int_M|\nabla f(x)|^2d\mu(x)}{\int_Mf(x)^2d\mu(x)}.
\end{equation*}
In the following, we only need to deal with Lipschitz functions as they are dense in $W^{1,2}(\mu)$.

The Bakry-\'{E}mery Ricci curvature tensor (see e.g. \cite{Bak}) is defined as
\begin{equation*}
 \mathrm{Ric}_{\mu}:=\mathrm{Ric}+\mathrm{Hess} V,
\end{equation*}
where $\mathrm{Ric}$ is the usual Ricci curvature tensor on $M$.

For any Borel subset $A\subseteq M$, its boundary measure $\mu^+(A)$ is defined as
\begin{equation*}
\mu^+(A):=\liminf_{r\rightarrow 0}\frac{\mu(O_r(A))-\mu(A)}{r},
\end{equation*}
where $O_r(A):=\{x\in M: d(x,a)<r \text{ for some } a\in A\}$ is the open $r$-neighborhood of $A$.
For any $A$ with $\mu(A)>0$, we define
\begin{equation*}
\phi(A):=\frac{\mu^+(A)}{\mu(A)}.
\end{equation*}
Given a nonnegative Lipschitz function $f: M\rightarrow \mathbb{R}_{\geq 0}$, we define
$M_f(t):=\{x\in M: f(x)> t\}$ and $\phi(f):=\inf_{t\in\mathbb{R}_{\geq 0}}\phi(M_f(t))$.
\begin{definition}[Multi-way isoperimetric constants \cite{Miclo2008,Funano2013}]\label{def:MultiwayCheeger}
 For a natural number $k$, the $k$-th isoperimetric constant is defined as
\begin{equation}\label{eq:MultiwayCheeger}
 h_k(M,\mu):=\inf_{A_0,A_1,\ldots,A_k}\max_{0\leq i\leq k}\phi(A_i),
\end{equation}
where the infimum is taken over all collections of $k+1$ non-empty, disjoint Borel subsets $A_0,A_1,\ldots,A_k$ of $M$ such that for each $0\leq i\leq k$, $\mu(A_i)>0$.
\end{definition}
Note that $h_k(M,\mu)\leq h_{k+1}(M, \mu)$ and that $h_1(M,\mu)$ is the classical Cheeger constant.
\begin{lemma}\label{lem:smallerLGT}
There exist two nonnegative disjointly supported Lipschitz functions $f_0$ and $f_1$ such that
\begin{equation}
\mathcal{R}(f_0)=\mathcal{R}(f_1)=\lambda_1(M, \mu).
\end{equation}
\end{lemma}
\begin{proof}
Let $f$ be the eigenfunction corresponding to $\lambda_1(M, \mu)$, and
\begin{equation*}
f_0(x):=\max\{f(x), 0\}, \,\, f_1(x):=\max\{-f(x), 0\}.
\end{equation*}
Since $\int_Mf(x)d\mu(x)=0$, we know that both $f_0$ and $f_1$ are non-trivial. By definition, we have,
\begin{align*}
&\int_M\langle\nabla f_i, \nabla f\rangle d\mu=\lambda_1(M, \mu)\int_Mf_ifd\mu=\lambda_1(M, \mu)\int_Mf_i^2d\mu\\
=&\int_M\langle\nabla f_i, \nabla f_i\rangle d\mu, \,\,\,\,\text{ for } i=0,1.
\end{align*}
This shows that $\mathcal{R}(f_i)=\lambda_1(M, \mu)$.
\end{proof}
A strong extension of this simple fact is the following lemma. It was first proved in the setting of finite graphs by Lee, Oveis Gharan and Trevisan \cite{LGT2013} and extended to Riemannian manifolds in \cite{Funano2013,Miclo2013}.
\begin{lemma}\label{lem:LGT}
For any natural number $k$, there exist $k+1$ nonnegative disjointly supported Lipschitz functions $f_0, f_1, \ldots, f_k$ and a numeric constant $C$ such that
for each $0\leq i\leq k$,
\begin{equation}
\mathcal{R}(f_i)\leq Ck^6\lambda_k(M,\mu).
\end{equation}
\end{lemma}

The following lemma is a direct corollary of the co-area inequality (Lemma 3.2 in Bobkov and Houdr\'{e} \cite{BH97}, see also \cite{Funano2013}).
\begin{lemma}\label{lem:cheeger-coarea}
For any nonnegative Lipschitz function $f$, we have
\begin{equation}
\phi(f)\leq \frac{\int_M|\nabla f(x)|d\mu(x)}{\int_Mf(x)d\mu(x)}.
\end{equation}
\end{lemma}
\begin{proof}
It is straightforward to observe
\begin{equation}\label{eq:Lebesgue}
\int_M f(x)d\mu(x)=\int_0^{\infty}\mu(M_f(t))dt.
\end{equation}
By the co-area inequality, we have
\begin{equation}\label{eq:coarea}
\int_0^{\infty}\mu^+(M_f(t))dt\leq \int_M|\nabla f(x)|d\mu(x).
\end{equation}
Combining (\ref{eq:Lebesgue}) and (\ref{eq:coarea}), we obtain
\begin{equation*}
\phi(f)\leq\frac{\int_0^{\infty}\mu^+(M_f(t))dt}{\int_0^{\infty}\mu(M_f(t))dt}\leq\frac{\int_M|\nabla f(x)|d\mu(x)}{\int_Mf(x)d\mu(x)}.
\end{equation*}
\end{proof}

\section{Proof of Theorem \ref{thm:ImproveCheeger}}\label{Section:proof}
In this section, we present an extension of the methods of Kwok et al. \cite{KLLGT2013} developed for finite graphs to prove Theorem \ref{thm:ImproveCheeger}. Particular efforts are made to improve the constant involved. Actually, we prove the following theorem.
\begin{thm}\label{thm:mainfunctionalversion}
For any nonnegative Lipschitz function $f: M\rightarrow \mathbb{R}$, we have
\begin{equation}
\phi(f)\leq 8\sqrt{2}k\frac{\mathcal{R}(f)}{\sqrt{\lambda_k(M, \mu)}}.
\end{equation}
\end{thm}
This theorem is an immediate consequence of the two lemmata below.

Given real values $t_0, t_1, \ldots, t_l\in \mathbb{R}$, we define a function $\psi_{t_0,t_1,\ldots, t_l}: \mathbb{R}\rightarrow \mathbb{R}$ as follows: for any $x\in \mathbb{R}$, let
\begin{equation*}
\psi_{t_0,t_1,\ldots, t_l}(x):=\arg\min_{t_i\in\{t_0,t_1,\ldots, t_l\}}|x-t_i|.
\end{equation*}
That is, $\psi_{t_0,t_1,\ldots, t_l}(x)\in \{t_0,t_1,\ldots, t_l\}$ is the closest value $t_i$ to $x$.

Given a nonnegative Lipschitz function $f$, and a sequence of real values $0=t_0\leq t_1\leq \cdots\leq t_{2k}=T:=\max_{x\in M}f(x)$, we have a $(2k+1)$-step function $g_k$ defined as
\begin{equation}
g_k(x):=\psi_{t_0,t_1,\ldots,t_{2k}}(f(x)).
\end{equation}
These step function approximations of $f$ share the following property.
\begin{lemma}
For any nonnegative Lipschitz function $f$, and a sequence of real values $0=t_0\leq t_1\leq \cdots\leq t_{2k}=T$, we have
\begin{equation}\label{eq:cheegerf-g}
\phi(f)\leq 8k\sqrt{\mathcal{R}(f)}\frac{\Vert f-g_k\Vert_{L^2(\mu)}}{\Vert f\Vert_{L^2(\mu)}}.
\end{equation}
\end{lemma}
\begin{proof}
We define a function $\eta :\mathbb{R}\rightarrow \mathbb{R}$ as follows:
\begin{equation*}
\eta(x):=|x-\psi_{t_0,t_1,\ldots,t_{2k}}(x)|, \,\,\forall x\in \mathbb{R}.
\end{equation*}
The function $h: M\rightarrow \mathbb{R}$ is then defined as
\begin{equation*}
h(x):=\int_0^{f(x)}\eta(t)dt,\,\,\forall x\in M.
\end{equation*}
Observe that for any two points $x,y\in M$, $h(x)\geq h(y)$ if and only if $f(x)\geq f(y)$. Hence we obtain $\phi(h)=\phi(f)$.
By Lemma \ref{lem:cheeger-coarea}, we have
\begin{equation}
\phi(h)\leq \frac{\int_M|\nabla h(x)|d\mu(x)}{\int_Mh(x)d\mu(x)}.
\end{equation}
We estimate the numerator as
\begin{align}
&\int_M|\nabla h(x)|d\mu(x)=\int_M\left|\nabla\left(\int_0^{f(x)}\eta(t)dt\right)\right|d\mu(x)\notag\\
=&\int_M|\nabla f(x)||\eta(f(x))|d\mu(x)
\leq\Vert |\nabla f| \Vert_{L^2(\mu)}\Vert f-g_k \Vert_{L^2(\mu)}.\label{eq:numerator}
\end{align}
Observe that, for any $x\in M$, we can find an integer $i$ such that $t_i<f(x)\leq t_{i+1}$, Hence, we have
\begin{align}
h(x)&=\sum_{j=0}^{i-1}\int_{t_j}^{t_{j+1}}\eta(t)dt+\int_{t_i}^{f(x)}\eta(t)dt\notag\\
&\geq \sum_{j=0}^{i-1}\frac{(t_{j+1}-t_j)^2}{4}+\frac{(f(x)-t_i)^2}{4}.\label{eq:hlower}
\end{align}
Using the Cauchy-Schwarz inequality, we obtain
\begin{align}
f^2(x)&=\left(\sum_{j=0}^{i-1}(t_{j+1}-t_j)+(f(x)-t_i)\right)^2\notag\\
&\leq 2k\sum_{j=0}^{i-1}(t_{j+1}-t_j)^2+2k(f(x)-t_i)^2.\label{eq:fupper}
\end{align}
Combining (\ref{eq:hlower}) and (\ref{eq:fupper}), we conclude that
\begin{equation}\label{eq:denominator}
h(x)\geq \frac{1}{8k}f^2(x), \,\,\forall x\in M.
\end{equation}
Now (\ref{eq:numerator}) and (\ref{eq:denominator}) together imply (\ref{eq:cheegerf-g}).
\end{proof}

In fact, one can find the step function approximation of $f$ with a very controlled behavior.
\begin{lemma}\label{lemma3.3}
For any nonnegative Lipschitz function $f: M\rightarrow \mathbb{R}$, there exists a $(2k+1)$-step approximation $g_k$ such that
\begin{equation}\label{eq:approximation}
\Vert f-g_k \Vert_{L^2(\mu)}^2\leq \frac{2}{\lambda_k(M,\mu)}\int_M|\nabla f(x)|^2d\mu(x).
\end{equation}
\end{lemma}

\begin{proof}
We construct $t_0,t_1,\ldots,t_{2k}$ inductively. First, set $t_0=0$. Suppose that we have already fixed the values of $t_0, t_1, \ldots, t_{i-1}$.
Then we find the value of $t_i$ in the following way. If there exists $t\geq t_{i-1}$ such that
\begin{equation}\label{eq:procedurebarrier}
\int_{\{x: t_{i-1}<f(x)\leq t\}}|f(x)-\psi_{t_{i-1},t}(f(x))|^2d\mu(x)=C_0:=\frac{\int_M|\nabla f(x)|^2d\mu(x)}{k\lambda_k(M,\mu)},
\end{equation}
we set $t_i$ to be the smallest one of such $t$; otherwise, we set $t_i=T$.

If $t_{2k}=T$, then the approximation $g_k:=\psi_{t_0,t_1,\ldots,t_{2k}}(f(x))$ would satisfy (\ref{eq:approximation}) and the proof is completed. It only remains to show that $t_{2k}<T$ leads to a contradiction.

Assume that $t_{2k}<T$. Then we can construct $2k$ nonnegative disjointly supported Lipschitz functions  $\{f^i\}_{i=1}^{2k}$ as follows:
\begin{equation*}
f^i(x):=\left\{
          \begin{array}{ll}
            |f(x)-\psi_{t_{i-1},t_i}(f(x))|, & \hbox{if $t_{i-1}<f(x)\leq t_i$;} \\
            0, & \hbox{otherwise.}
          \end{array}
        \right.
\end{equation*}
Since we have for any $x,y$ with $t_{i-1}<f(x),f(y)\leq t_i$,
\begin{equation*}
 \left||f(x)-\psi_{t_{i-1},t_i}(f(x))|-|f(y)-\psi_{t_{i-1},t_i}(f(y))|\right|\leq |f(x)-f(y)|,
\end{equation*}
we conclude that
\begin{equation}
 \sum_{i=1}^{2k}\int_M|\nabla f^i(x)|^2d\mu(x)\leq\int_M|\nabla f(x)|^2d\mu(x).
\end{equation}
Observing $\Vert f^i\Vert^2_{L^2(\mu)}=C_0$ by (\ref{eq:procedurebarrier}), we obtain
\begin{equation}\label{eq:forClaim}
 \sum_{i=1}^{2k}\mathcal{R}(f^i)\leq \frac{1}{C_0}\int_M|\nabla f(x)|^2d\mu(x)=k\lambda_k(M,\mu).
\end{equation}
The last equality above follows directly from the definition of $C_0$.
\begin{claim}\label{claim}
 We can find $k+1$ functions from $\{f^i\}_{i=1}^{2k}$, denoted by $f^1,\ldots, f^{k+1}$ after relabeling, such that
\begin{equation*}
  \mathcal{R}(f^j)<\lambda_k(M,\mu), \,\,j=1,2,\ldots,k+1.
\end{equation*}
\end{claim}
If this claim were false, then there would exist at least $k$ functions of $\{f^i\}_{i=1}^{2k}$, say $f^{i_1},\ldots, f^{i_{k}}$, such that
\begin{equation*}
 \mathcal{R}(f^{i_l})\geq \lambda_k(M,\mu), \,\,l=1,2,\ldots,k.
\end{equation*}
Together with (\ref{eq:forClaim}), this would imply that $k$ functions of $\{f^i\}_{i=1}^{2k}$ have vanishing Rayleigh quotients. Hence $f^i$ is a constant function for some $i$. By definition, this can only be the zero function. This contradicts the fact that $f$ is Lipschitz continuous.

Now by the min-max principle (see e.g. \cite[III 27]{Berard86}), we have
\begin{align}
 \lambda_k(M,\mu)\leq \sup_{(\alpha_j)\in \mathbb{R}^{k+1}}\mathcal{R}\left(\sum_{j=1}^{k+1}\alpha_jf^j\right)
 =\max_{1\leq j\leq k+1}\mathcal{R}(f^j)<\lambda_k(M,\mu),\label{eq:minmax}
\end{align}
which is the required contradiction.
\end{proof}
\begin{proof}[Proof of Theorem \ref{thm:ImproveCheeger}]
We apply Theorem \ref{thm:mainfunctionalversion} to the two functions provided by Lemma \ref{lem:smallerLGT}. Recalling the definition \ref{eq:MultiwayCheeger}) for the Cheeger constant $h_1(M,\mu)$, we prove (\ref{eq:ImprovedCheeger}).
\end{proof}
Employing Lemma \ref{lem:LGT} instead of Lemma \ref{lem:smallerLGT}, we arrive at the following extension of Corollary 1 (i) in \cite{KLLGT2013}. It is considered to be an improved version of the higher-order Cheeger inequality in \cite{LGT2013,Funano2013,Miclo2013}.

\begin{thm}\label{thm:ImprovedHigher}
For any closed weighted Riemannian manifold $(M,\mu)$ and any two natural numbers $l,k$, there exists a numeric constant $C>0$ such that
\begin{equation}
 h_k(M,\mu)\leq Clk^6\frac{\lambda_k(M,\mu)}{\sqrt{\lambda_l(M, \mu)}}.
\end{equation}
\end{thm}

\section{Applications and extensions}
\subsection{Multi-way isoperimetric constants} Theorem \ref{thm:Main} has very interesting applications. Combining it with Buser-Ledoux inequality (\ref{eq:LedouxBuser}), we obtain the following improvement of Funano's higher-order Buser-Ledoux inequality (Theorem 1.7 in \cite{Funano2013}).
\begin{thm}\label{thm:higerBuserLedoux}
 On a closed Riemannian manifold $(M,\mu)$ of nonnegative Bakry-\'{E}mery Ricci curvature, we have for any natural number $k$
\begin{equation}\label{eq29}
 h_k(M,\mu)\geq h_1(M,\mu)\geq\frac{(e-1)^2}{16\sqrt{2}e^2}\frac{1}{k}\sqrt{\lambda_k(M,\mu)}.
\end{equation}
\end{thm}
We remark that the constant above $\frac{(e-1)^2}{16\sqrt{2}e^2k}\geq \frac{1}{57k}$ improves the $\frac{1}{80k^3}$ in \cite{Funano2013}.
\begin{proof}
The formula (\ref{eq29}) follows from
\begin{equation*}
 h_k(M,\mu)\geq h_1(M,\mu)\geq \frac{e-1}{\sqrt{2}e}\sqrt{\lambda_1(M,\mu)}\geq\frac{e-1}{\sqrt{2}e}\cdot\frac{e-1}{16ek}\sqrt{\lambda_k(M,\mu)},
\end{equation*}
where we used monotonicity of $h_k(M,\mu)$, Theorem \ref{thm:LedouxBuser} and Theorem \ref{thm:Main}.
\end{proof}

We obtain the following multi-way isoperimetric constant ratio estimate, which improves Funano's exponential upper bound (Theorem 1.6 in \cite{Funano2013}) to a polynomial bound.
\begin{thm}\label{thm:MIsoperiratio}
 There exists a universal numeric constant $C>0$ such that if $(M, \mu)$ is a closed weighted Riemannian manifold of nonnegative Bakry-\'{E}mery Ricci curvature and $k$ is a natural number, then we have
\begin{equation}\label{eq:Isoperratio}
 h_k(M,\mu)\leq Ck\sqrt{\log(1+k)}h_1(M,\mu).
\end{equation}
\end{thm}
\begin{proof}
We need the following shifted higher-order Cheeger inequality (Theorem 4.1 in \cite{LGT2013} and Theorem 12 in \cite{Miclo2013})
\begin{equation}\label{eq:shiftedCheeger}
 h_k(M,\mu)\leq C_1\sqrt{\lambda_{2k}(M,\mu)\log(1+k)}
\end{equation}
with some universal numeric constant $C_1>0$.
Combining (\ref{eq:shiftedCheeger}) with Theorem \ref{thm:Main} and Theorem \ref{thm:LedouxBuser}, we prove (\ref{eq:Isoperratio}).
\end{proof}
\begin{remark}
The shifted higher-order Cheeger inequality (\ref{eq:shiftedCheeger}) was proved in finite graph setting by Lee, Oveis Gharan and Trevisan \cite{LGT2013}. This was then extended to Riemannian setting via a probability theoretic approximation procedure by Miclo \cite{Miclo2013}. However, a direct proof of it in Riemannian setting seems to be still unknown (see \cite{Funano2013}).
\end{remark}
\begin{remark}
 In the above two results, estimates of the two ratios, $$\sqrt{\lambda_k(M,\mu)}/h_k(M,\mu)\text{  and  }h_k(M,\mu)/h_1(M,\mu),$$ are improved. However, the author does not know whether they have the optimal order of $k$ or not (see Question 6.3 in Funano \cite{Funano2013}). Observe in Example \ref{example} that we have $h_k(T_L^1)=2k/L$ (by an argument similar to Proposition 7.3 in \cite{Liu13}). Therefore for this example, the first ratio is of order $0$ and the second is of order $1$ in $k$.
\end{remark}

\subsection{Observable diameter}
Observable diameter with parameter $\kappa>0$ is an important concept introduced by Gromov \cite{Gromov99}. For $(M,\mu)$, the partial diameter $\text{diam}(\mu; 1-\kappa)$ is defined as the infimum of the diameter of $A$ over all Borel subsets $A\subseteq M$ with $\mu(A)\geq 1-\kappa$. The observable diameter is then defined as
\begin{equation*}
 \text{ObsDiam}((M,\mu);-\kappa):=\sup\{\text{diam}(f_*\mu;1-\kappa)\},
\end{equation*}
where the supremum is taken over all $1$-Lipschitz functions $f: M\rightarrow \mathbb{R}$. This constant is important for characterizing measure concentration phenomena. Funano \cite{Funano2013} obtained the following version of the Gromov-Milman \cite{GM83} inequality:
\begin{equation}\label{eq:GromovMilman}
 \text{ObsDiam}((M,\mu);-\kappa)\leq \frac{6}{\sqrt{\lambda_1(M,\mu)}}\log\frac{2}{\kappa}.
\end{equation}
Combining (\ref{eq:GromovMilman}) with Theorem \ref{thm:Main} provides the following result, which can be considered as a dimension-free Cheng's inequality.
\begin{thm}\label{thm:higerGromovMilman}
For any closed weighted Riemannian manifold $(M,\mu)$ of nonnegative Bakry-\'{E}mery Ricci curvature and any natural number $k$, we have
\begin{equation}\label{eq:dimenfreeCheng}
 \text{ObsDiam}((M,\mu);-\kappa)\leq \frac{152k}{\sqrt{\lambda_k(M,\mu)}}\log\frac{2}{\kappa}.
\end{equation}
\end{thm}
Note that $152k$ here improves the constant $e^{ck}$ in Corollary 1.3 of \cite{Funano2013}. Recall that assuming nonnegativity of Ricci curvature, Cheng (Corollary 2.2 in \cite{Cheng75}) proved
\begin{equation}\label{eq:cheng75}
\text{diam}(M)\leq \frac{\sqrt{2n(n+4)}k}{\sqrt{\lambda_k(M)}},
\end{equation}
for any natural number $k$,
where $n$ is the dimension of $M$.
Kei Funano pointed out to us that (\ref{eq:dimenfreeCheng}) implies (\ref{eq:cheng75}) up to a universal constant, using the equivalence of observable diameter and separation distance (see e.g. Lemma 5.6 in \cite{Funano2013}) and an argument of Chung, Grigor'yan and Yau (see (3.15) in \cite{CGY}). Observe that one can not expect an exact dimension-free version of (\ref{eq:cheng75}) in view of the example of unit spheres $\{S^n\}_{n=1}^{\infty}$. Actually, we have $\lambda_1(S^n)=n$ \cite{GM83}.


\subsection{Extensions to Alexandrov spaces}\label{subsection:Exten}
An Alexandrov space $(X,d)$ of curvature bounded from below is a complete geodesic metric space which satisfies locally the Toponogov triangle comparison theorem for sectional curvature.
The Hausdorff dimension of an Alexandrov space is an integer or infinity. See \cite{BBI01,BGP,OS94} for more details about Alexandrov geometry. In this subsection we consider a compact $n$-dimensional Alexandrov space $X$ equipped with the $n$-dimensional Hausdorff measure $\mathcal{H}^n$. A detailed review of the Sobolev space $W^{1,2}(\mathcal{H}^n)$ of an Alexandrov space can be found in \cite[Section 2]{GKK13}. Note that the set of Lipschitz functions is dense in $W^{1,2}(\mathcal{H}^n)$ (see \cite[Theorem 1.1]{KMS01}). For a Lipshitz function $f: X\rightarrow \mathbb{R}$, the weak gradient vector $\nabla f(x)$ satisfies
\begin{equation}\label{eq:gradient}
 \sqrt{\langle\nabla f(x), \nabla f(x)\rangle}=\limsup_{y\rightarrow x}\frac{|f(x)-f(y)|}{d(x,y)}, \,\,a.e. \,\,x.
\end{equation}
(See \cite{KMS01} for precise definitions of the weak gradient vector and the inner product $\langle\cdot, \cdot\rangle$.)
The Dirichlet energy is defined as $$\mathcal{E}(f,g):=\int_X \langle\nabla f, \nabla g\rangle d\mathcal{H}^n,\,\,\text{ for }f,g\in W^{1,2}(\mathcal{H}^n).$$
This definition coincides with Cheeger's energy functional \cite{Cheeger99} in terms of minimal generalized upper gradient.

The spectral theory of the Laplacian $\Delta_X$ associated to $(\mathcal{E}, W^{1,2}(\mathcal{H}^n))$ on an Alexandrov space was studied in \cite{KMS01}. In particular, $\Delta_X$ has discrete spectrum consisting of eigenvalues
\begin{equation*}
 0=\lambda_1(X)<\lambda_2(X)\leq \cdots\leq\lambda_k(X)\leq\cdots\nearrow\infty
\end{equation*}
with finite multiplicity. Moreover, the corresponding eigenfunctions are Lipschitz continuous (Theorem 4.4 of \cite{GKK13}). The min-max principle we used in (\ref{eq:minmax}) generalizes to the setting of Alexandrov spaces via a standard argument.

Note that the co-area inequality of Bobkov and Houdr\'{e} \cite{BH97} holds in a general metric measure space where the measure is not atomic, and the modulus of a function's gradient there coincides with (\ref{eq:gradient}).
Bearing in mind those facts mentioned above,
we see that Theorem \ref{thm:ImproveCheeger} extends readily to a compact finite-dimensional Alexandrov space of curvature bounded from below.

As pointed out by Funano (Remark 4.5 in \cite{Funano2013}), employing the Bakry-\'{E}mery type gradient estimate for the heat kernel on an Alexandrov space due to Gigli, Kuwada and Ohta \cite{GKK13}, we conclude that Theorem \ref{thm:LedouxBuser} holds also on a compact finite-dimensional Alexandrov space of nonnegative curvature. This then confirms Theorem \ref{thm:AleNonnegative} in the setting of Alexandrov spaces.

We remark that Theorem \ref{thm:AleNonnegative} is still valid for a weighted Alexandrov space satisfying $CD(0,\infty)$ in the sense of Lott-Sturm-Villani (\cite{LV,Sturm1}), since Gigli-Kuwada-Ohta gradient estimate extends to this setting (see \cite{GKK13}).


\section*{Acknowledgements}
I am very grateful to Norbert Peyerimhoff for many helpful discussions. I thank Kei Funano, Alexander Grigor'yan, J\"{u}rgen Jost, and  Feng-Yu Wang for their interest and valuable comments. Special thanks go to Daniel Grieser for pointing out Example \ref{exampleThinTorus} and Martin Kell for pointing out the possibility of extending the results to compact metric measure spaces satisfying the Riemannian curvature-dimension conditions $RCD^*(0, N)$ (see \cite{EKS}).
This work was supported by the EPSRC Grant EP/K016687/1.

\end{document}